%% file: main.tex
\documentclass{article}

\input{preamble.tex}

\input{title}

\begin{document}


\maketitle

\begin{abstract}

\input{abstract}

\end{abstract}

\section{Introduction}\label{sec:intro}
\input{introduction}

\section{Notation}\label{sec:notation}
\input{notation}

\section{Sufficient Condition}\label{sec:sufficient}

\input{sufficient}

\section{Parametrization of SVDs}\label{sec:params_of_SVDs}
\input{SVD}

\section{Necessary Condition}\label{sec:necessary}
\input{necessary}

\section{Another Equivalent Conditions}\label{sec:commuting} 
\input{commuting}

\section{Miscellaneous}\label{sec:misc}
\input{misc}

\section{Geometric Necessary Condition}
\input{geometric}

\section{SVD of a Projection}
\input{SVD_proj}

\section{MATLAB Code}\label{sec:matlab}
\input{matlab}

\bibliographystyle{unsrt}  
\bibliography{references}  

\end{document}

%% file: preamble.tex
\usepackage{arxiv}

\usepackage[utf8]{inputenc} 
\usepackage[T1]{fontenc}    
\usepackage{hyperref}       
\usepackage{url}            
\usepackage{booktabs}       
\usepackage{amsfonts}       
\usepackage{nicefrac}       
\usepackage{microtype}      
\usepackage{lipsum}
\usepackage{mathtools,physics}
\usepackage{enumerate}
\usepackage{graphicx}
\graphicspath{ {./images/} }
\usepackage{amsmath}
\usepackage{amssymb}
\usepackage{amsthm}
\usepackage{stackengine} 
\stackMath
\usepackage[framed,autolinebreaks,useliterate]{mcode} 
\usepackage{centernot}
\usepackage{caption,makecell} 
\makeatletter
\let\save@mathaccent\mathaccent
\newcommand*\if@single[3]{%
  \setbox0\hbox{${\mathaccent"0362{#1}}^H$}%
  \setbox2\hbox{${\mathaccent"0362{\kern0pt#1}}^H$}%
  \ifdim\ht0=\ht2 #3\else #2\fi
  }
\newcommand*\rel@kern[1]{\kern#1\dimexpr\macc@kerna}
\newcommand*\widebar[1]{\@ifnextchar^{{\wide@bar{#1}{0}}}{\wide@bar{#1}{1}}}
\newcommand*\wide@bar[2]{\if@single{#1}{\wide@bar@{#1}{#2}{1}}{\wide@bar@{#1}{#2}{2}}}
\newcommand*\wide@bar@[3]{%
  \begingroup
  \def\mathaccent##1##2{%
    \let\mathaccent\save@mathaccent
    \if#32 \let\macc@nucleus\first@char \fi
    \setbox\z@\hbox{$\macc@style{\macc@nucleus}_{}$}%
    \setbox\tw@\hbox{$\macc@style{\macc@nucleus}{}_{}$}%
    \dimen@\wd\tw@
    \advance\dimen@-\wd\z@
    \divide\dimen@ 3
    \@tempdima\wd\tw@
    \advance\@tempdima-\scriptspace
    \divide\@tempdima 10
    \advance\dimen@-\@tempdima
    \ifdim\dimen@>\z@ \dimen@0pt\fi
    \rel@kern{0.6}\kern-\dimen@
    \if#31
      \overline{\rel@kern{-0.6}\kern\dimen@\macc@nucleus\rel@kern{0.4}\kern\dimen@}%
      \advance\dimen@0.4\dimexpr\macc@kerna
      \let\final@kern#2%
      \ifdim\dimen@<\z@ \let\final@kern1\fi
      \if\final@kern1 \kern-\dimen@\fi
    \else
      \overline{\rel@kern{-0.6}\kern\dimen@#1}%
    \fi
  }%
  \macc@depth\@ne
  \let\math@bgroup\@empty \let\math@egroup\macc@set@skewchar
  \mathsurround\z@ \frozen@everymath{\mathgroup\macc@group\relax}%
  \macc@set@skewchar\relax
  \let\mathaccentV\macc@nested@a
  \if#31
    \macc@nested@a\relax111{#1}%
  \else
    \def\gobble@till@marker##1\endmarker{}%
    \futurelet\first@char\gobble@till@marker#1\endmarker
    \ifcat\noexpand\first@char A\else
      \def\first@char{}%
    \fi
    \macc@nested@a\relax111{\first@char}%
  \fi
  \endgroup
}
\makeatother

\newtheorem{theorem}{Theorem}[section]
\newtheorem{lemma}[theorem]{Lemma}

\newtheorem{corollary}[theorem]{Corollary}
\newtheorem{remark}[theorem]{Remark}
\newtheorem{definition}[theorem]{Definition}
\newtheorem{example}[theorem]{Example}

\newcommand{\rng}[1]{\mathcal{C}(#1)}
\newcommand{\nullsp}[1]{\mathcal{N}(#1)}
\newcommand{\pinv}[1]{{#1}^+\!}

\newcommand{\minv}[2]{{#1}\{{#2}\}\!}

\DeclarePairedDelimiter\brac{(}{)}
\DeclarePairedDelimiter\bracc{\{}{\}}

\newcommand{\bK}{\mathbb{K}}
\newcommand{\bC}{\mathbb{C}}
\newcommand{\bR}{\mathbb{R}}

\newcommand{\cN}{\mathcal{N}}

\newcommand{\diag}{\operatorname{diag}}
\newcommand{\linsp}{\operatorname{span}}

\numberwithin{equation}{section}

%% file: title.tex
\title{Characterization of Matrices Satisfying the Reverse Order Law for the Moore--Penrose Pseudoinverse.}

\author{
  Oskar Kędzierski\thanks{the author also works at the Faculty of Mathematics, Informatics, and Mechanics of the University of Warsaw, \texttt{oskar@mimuw.edu.pl}}\\
  NASK National Research Institute\\
  Warsaw, Poland\\
  \texttt{oskar.kedzierski@nask.pl}\\
}

%% file: abstract.tex
We give a constructive characterization of matrices satisfying the reverse-order law for the Moore--Penrose pseudoinverse. In particular, for a given matrix $A$ we construct another matrix $B$, of arbitrary compatible size and chosen rank, in terms of the right singular vectors of $A$, such that the reverse order law for $AB$ is satisfied. Moreover, we show that any matrix satisfying this law comes from a similar construction. 
As a consequence, several equivalent conditions to $\pinv{B}\pinv{A}$ being a pseudoinverse of $AB$ are given, for example $\rng{A^*AB}=\rng{BB^*A^*}$ or $B\pinv{\brac{AB}}A$ being an orthogonal projection.
In addition, we parameterize all possible SVD decompositions of a fixed matrix and give Greville--like equivalent conditions for $\pinv{B}\pinv{A}$ being a $\{1,2\}$-,$\{1,2,3\}$- and $\{1,2,4\}$-inverse of $AB$ with a geometric insight in terms of the principal angles between $\rng{A^*}$ and $\rng{B}$. 

%% file: introduction.tex
Penrose~\cite{Penrose_1955} has proven that for any matrix $A\in\bK^{m\times n}$, where $\bK=\bR$ or $\bC$, there exists a unique matrix $X\in\bK^{n\times m}$ such that

\begin{enumerate}[i)]
    \item 
    \begin{equation}\label{eq:Penrose_cond1} AXA=A, \end{equation}
    \item 
    \begin{equation}\label{eq:Penrose_cond2} XAX=X, \end{equation}
    \item 
    \begin{equation}\label{eq:Penrose_cond3} (AX)^*=AX, \end{equation}
    \item 
    \begin{equation}\label{eq:Penrose_cond4} (XA)^*=XA. \end{equation}
\end{enumerate}

Such a matrix is an example of a generalized inverse and is denoted $X=\pinv{A}$ and referred to usually as the Moore--Penrose pseudoinverse or just pseudoinverse. It shares some properties of the inverse for the non--singular square matrix. However, the reverse order law, i.e., $\pinv{\brac*{AB}}=\pinv{B}\pinv{A}$ does not hold for arbitrary matrices, for example
\[A=\left[\begin{matrix}1 & 1\end{matrix}\right],\quad B=\left[\begin{matrix}1\\0\end{matrix}\right],\quad \pinv{\brac*{AB}}=1,\quad \pinv{B}\pinv{A}=\frac{1}{2}.\]

It is well--known that the reverse order law holds in the following cases:

\begin{enumerate}[i)]
    \item $A^*A$ commutes with $BB^*$,
    \item $A$ is of full column rank and $B$ is of full row rank,
    \item $\rng{A^*}=\rng{B}$.
\end{enumerate}
Note that condition i) is satisfied, for example, when $A^*A=I$ (matrix $A$ has orthonormal columns) or $BB^*=I$ (matrix $B$ has orthonormal rows) or $AB=0$.
Neither of conditions $i)- iii)$ is necessary, which follows from the following counterexample.

\[A=\left[\begin{array}{cccc} 1 & 0 & 0 & 0\\ 0 & 2 & 0 & 0\\ 0 & 0 & 1 & 0 \end{array}\right],\quad B=\left[\begin{array}{cccc} 2 & 4 & 3 & 2\\ 2 & 4 & 1 & 2\\ 0 & 0 & 0 & 0\\ 0 & 0 & 0 & 0 \end{array}\right],\quad \pinv{\brac*{AB}}=\pinv{B}\pinv{A}=\frac{1}{48}\left[\begin{array}{rrr} -2 & 3 & 0\\ -4 & 6 & 0\\ 24 & -12 & 0\\ -2 & 3 & 0 \end{array}\right],\]
\[\rank A=3,\quad \rank B=2,\quad \left[A^*A,BB^*\right]=81\left[\begin{array}{rrrr} 0 & -1 & 0 & 0\\ 1 & 0 & 0 & 0\\ 0 & 0 & 0 & 0\\ 0 & 0 & 0 & 0 \end{array}\right],\quad AB\neq 0. \]





The sufficient and necessary conditions were given by Greville, see~\cite{Greville_note} and \cite[p.~160, Ex.~22]{BenGreville}

\begin{theorem}[Greville]\label{thm:greville_conds}
The following conditions are equivalent.
\begin{enumerate}[i)]
    \item $\pinv{\brac*{AB}}=\pinv{B}\pinv{A}$,
    \item $\rng{A^*AB}\subset\rng{B}$ and $\rng{BB^*A^*}\subset\rng{A^*}$,
    \item $\pinv{A}ABB^*A^*AB\pinv{B}=BB^*A^*A$.
\end{enumerate}
    
\end{theorem}

There is no straightforward way of constructing a concrete example of matrices with given ranks and sizes satisfying the reverse order law from the Greville's conditions.

In the following paper, we give a sufficient condition (Theorem~\ref{thm:sufficient}) and a related necessary condition (Theorem~\ref{thm:necessary})which allow, for a fixed matrix $A$, to explicitly construct all matrices $B$ satisfying the reverse order law, or vice versa. This is done in terms of the singular value decompositions of matrices $A$ and $B$.

This generalizes the result of Schwerdtfeger~\cite[p.~326]{Schwerdtfeger} stated only for matrices of equal ranks and only as a sufficient condition. We give independent proofs and geometric interpretations to already known conditions for $\pinv{B}\pinv{A}$ being $\{1,2\}$-inverse of $AB$, see Tian~[Theorem~9.1]\cite{Tian512} and for the reverse order law, see Tian~[Theorem~11.1]\cite{Tian512}. In particular, we show that the following conditions are equivalent.

\begin{theorem}
Let $A\in\bK^{m\times n},B\in\bK^{n\times k}$. 
    \begin{enumerate}[i)]
    \item $\pinv{\brac*{AB}}=\pinv{B}\pinv{A}$,
    \item there exist essential reduced SVD decompositions $A=U_A\Sigma_A V_A^*$ and $B=U_B\Sigma_B V_B^*$ such that
    \[V_A^*U_B=\left[\begin{array}{cc}
            Q & 0\\
            0 & 0\\
        \end{array}\right],\]
        where $Q$ is an unitary matrix or $V_A^*U_B=0$,     
        \item $\rng{A^*AB}=\rng{BB^*A}$,
        \item $\pinv{\brac*{\brac*{A^*A}\brac*{BB^*}}}=\pinv{\brac*{BB^*}}\pinv{\brac*{A^*A}}$,
    \item $\pinv{A}A$ commutes with $BB^*$ and $B\pinv{B}$ commutes with $A^*A$.
\end{enumerate}
\end{theorem}

Moreover, there is a similarity between the conditions on $\pinv{B}\pinv{A}$ being an $\{1,2\}-$
and $\{1,2,3,4\}-$ inverse of $AB$, see Tables~\ref{tab:similar_equivalent} and \ref{tab:similar_implied}, also Theorem~\ref{thm:123_124_equiv}.

\renewcommand{\arraystretch}{1.4}
\begin{table}[h]
    \centering
    \begin{tabular}{c|c}
       $\{1,2\}$-ROL  & $\{1,2,3,4\}$-ROL \\ \hline
       $\rng{\pinv{A}AB}\subset\rng{B}$ or $\rng{B\pinv{B}\pinv{A}}\subset\rng{\pinv{A}}$  & $\rng{A^*AB}\subset\rng{B}$ and $\rng{BB^*A^*}\subset\rng{A^*}$ \\
        $\rng{\pinv{A}AB}=\rng{B\pinv{B}\pinv{A}}$ & $\rng{A^*AB}=\rng{BB^*A^*}$\\
      $\pinv{A}AB\pinv{B}\pinv{A}AB\pinv{B}=B\pinv{B}\pinv{A}A$   & $\pinv{A}ABB^*A^*AB\pinv{B}=BB^*A^*A$ \\
      $\pinv{A}A,B\pinv{B}$ commute   & $\pinv{A}A,BB^*$ commute and $B\pinv{B},A^*A$ commute \\
       $\pinv{A}AB\pinv{B}$ or $B\pinv{B}\pinv{A}A$ is an orthogonal projection & $B\pinv{\brac{AB}}A\text{ is an orthogonal projection}$ \\
       $\{1,2\}$-ROL holds for ${\pinv{A}A},B\pinv{B}$ or for ${A^*A},BB^*$ &  $\{1,2,3,4\}$-ROL holds for ${A^*A},BB^*$ \\
      \makecell{the principal angles between $\rng{A^*}$ and $\rng{B}$ \\ belong to the set $\bracc*{0,\frac{\pi}{2}}$}  & \makecell{the principal angles between $\rng{A^*}$ and $\rng{B}$ \\ belong to the set $\bracc*{0,\frac{\pi}{2}}$ and  $\rng{A^*}\cap\rng{B}$ \\  is spanned by left singular vectors of $A$ \\  and right singular vectors of $B$ in some \\  SVD decompositions of $A$ and $B$}\\

    \end{tabular}
    \caption{Similarity between equivalent conditions}
    \label{tab:similar_equivalent}
\end{table}

\renewcommand{\arraystretch}{1.4}
\begin{table}[h]
    \centering
    \begin{tabular}{c|c}
        $\{1,2\}$-ROL  & $\{1,2,3,4\}$-ROL \\ \hline
       $\rank AB=\dim(\rng{A^*}\cap\rng{B})$  & $\rank AB=\dim(\rng{A^*}\cap\rng{B})$\\
        $\rng{\pinv{A}AB}=\rng{B\pinv{B}\pinv{A}}=\rng{A^*}\cap\rng{B}$ & $\rng{A^*AB}=\rng{BB^*A^*}=\rng{A^*}\cap\rng{B}$\\
        other possible $\{1,2\}$-ROL hold &  other possible $\{1,2,3,4\}$-ROL hold\\
    \end{tabular}
    \caption{Similarity between conditions implied by ROL}
    \label{tab:similar_implied}
\end{table}

I would like to thank Michał Karpowicz for introducing me to this topic and thank my co-workers, especially Aikaterini Aretaki and Antonina Krajewska, and the NASK administration for creating a stimulating and friendly working environment.

%% file: notation.tex
The set of all matrices with $m$ rows, $n$ columns, and coefficients in field $\bK$ is denoted $\bK^{m\times n}$, where $\bK=\bR$ or $\bK=\bC$.  The column space of matrix $A$ is denoted by $\rng{A}$ and the null space by $\nullsp{A}$. The linear span of vectors $v_1,\ldots,v_n$, i.e., the set of all linear combinations of $v_1,\ldots,v_n$ over $\bK$ is denoted by $\linsp(v_1,\ldots,v_n)$. The Moore--Penrose pseudoinverse of matrix $A$ is denoted by $\pinv{A}$. The complex conjugate is denoted by $A^*$. Matrix is Hermitian (or symmetric over $\bR$) if $A^*=A$. A square matrix is unitary (or orthogonal over $\bR$) if $A^*A=I$ where $I$ denotes the unit matrix. SVD stands for singular value decomposition, i.e., decomposition $A=U\Sigma V^*$, where $U,V$ are unitary matrices, and $\Sigma$ is a real non--negative generalized diagonal matrix with decreasing diagonal entries. If $\rank A=r$ then the first $r$ columns of matrix $U$ are called left singular vectors of $A$ and likewise the first $r$ columns of matrix $V$ are called right singular vectors of $A$. Singular value decomposition of a fixed matrix $A$ is not unique, except for the matrix $\Sigma$. For example, $I=QIQ^*$ for any unitary matrix $Q$. For any matrix $A\in\bK^{m\times n}$, the matrix $A_{p:q,s:t}$ is a submatrix of $A$ consisting of rows from $p$ to $q$ and columns from $s$ to $t$, both $q$ and $t$ included (MATLAB notation). A single colon with no numbers indicates that there are no restrictions. Furthermore, for $I\subset\{1,\ldots, n\}$ the matrix $A_I\in\bK^{m\times\abs{I}}$ is a submatrix of the matrix $A$ consisting of columns indexed by $I$. The unitary group is denoted by $U(n,\bK)$. 

%% file: sufficient.tex
\begin{theorem}\label{thm:sufficient}
  Let $A\in\bK^{m\times n},\ B\in\bK^{n\times k}$ be two matrices of ranks $r_A, r_B$ with  given SVD decompositions $A=U_A\Sigma_A V_A^*$ and $B=U_B\Sigma_B V_B^*$. 
 If there exists a (possibly empty) set $J\subset\{1,\ldots,r_A\}$, indexing the first $r_A$ columns of matrix $V_A$ (i.e., the right singular vectors of $A$), such that
\begin{equation}\label{eq:spaces_suff}
\linsp \left(u_{B,i}\in\bK^n \mid u_{B,i}\notin\nullsp{A},\ i\in\{1,\ldots,r_B\}\right)=\rng{\left(V_A\right)_J},
\end{equation}
where $u_{B,i}$ denotes the $i$-th column of matrix $U_B$, then
    \[\pinv{(AB)}=\pinv{B}\pinv{A}.\]  
\end{theorem}
\begin{proof}
    If  the set $J$ is empty then $AB=0$ and $\pinv{B}\pinv{A}=0$. Assume that $J\neq\emptyset$. For simplicity, let $v_i=v_{A,i}\in\bK^n$ denote the $i$-th column of matrix $V_A$ and let $u_i=u_{B,i}\in\bK^n$ denote the $i$-th column of matrix $U_B$. Set $r_A=\rank A,r_B=\rank B$. By the SVD decompositions
    \[A^*A=\sum_{j=1}^{r_A} \sigma^2_{A,j} v_j v_j^*,\quad BB^*=\sum_{j=1}^{r_B} \sigma^2_{B,j} u_j u_j^*.\]
    Conditions 
    \[\rng{A^*AB}\subset\rng{B}\quad\text{and} \quad\rng{BB^*A^*}\subset\rng{A^*},\]
    are equivalent to
    
    \begin{equation}\label{eq:GA}
        u_{i_2}^*\sum_{j=1}^{r_A} \sigma^2_{A,j} v_j v_j^* u_{i_1}=0,
    \end{equation}
    for $i_1=1,\ldots,r_B,\ i_2=r_B+1,\ldots,n$ and
    \begin{equation}\label{eq:GB}
        v_{i_2}^*\sum_{j=1}^{r_B} \sigma^2_{B,j} u_j u_j^* v_{i_1}=0,
    \end{equation}
    for $i_1=1,\ldots,r_A,\ i_2=r_A+1,\ldots,n$.

    Consider first Equation~(\ref{eq:GA}). If $u_{i_1}\in \nullsp{A}$ then
    $\sum_{j=1}^{r_A} \sigma^2_{A,j} v_j v_j^* u_{i_1}=0$ since $v_j\in \rng{A^*}$ for any $j=1,\ldots,r_A$. Otherwise $u_{i_1}=\sum_{j\in J}\alpha_j v_j$ for some $\alpha_j\in\bK$ by assumption~(\ref{eq:spaces_suff}) and, consequently
    \[\sum_{j=1}^{r_A} \sigma^2_{A,j} v_j v_j^* u_{i_1}=\sum_{j\in J} \sigma^2_{A,j} a_j v_j\in \rng{B}, \]
    by orthogonality of $v_i$'s and the main assumption~(\ref{eq:spaces_suff}) (that is $v_j$ for $j\in J$ belongs to the left hand side of~(\ref{eq:spaces_suff}) and hence $v_j\in\rng{B}$). Therefore, Equation~(\ref{eq:GA}) is satisfied.

    To show that Equation~(\ref{eq:GB}) is satisfied, note that $v_{i_2}\in\nullsp{A}$ and all $u_j$'s which are not in $\nullsp{A}$ are orthogonal to $v_{i_2}$ (by assumption~(\ref{eq:spaces_suff}), they belong to the left hand side and therefore to the right hand side, which is given by the linear span of vectors $v_j\in\rng{A^*},j\in J$). Hence
   \[v_{i_2}^*\sum_{j=1}^{r_B} \sigma^2_{B,j} u_j u_j^* v_{i_1}=\sum_{u_j\in\nullsp{A}} \sigma^2_{B,j} \left(v_{i_2}^*u_j\right) u_j^* v_{i_1}=0,\]
    as $v_{i_1}\in\rng{A^*}$. 
    \end{proof}

\begin{remark}
    By the necessary condition, cf.~Theorem~\ref{thm:necessary}, the set $J$ can be given in more concrete terms. 
\end{remark}

\begin{remark}
    For a fixed matrix $A$ with a given SVD decomposition $A=U_A\Sigma_A V_A^*$ and a given set $J\subset\{1,\ldots,r_A\}$ matrix $B$ that satisfies the assumptions of Theorem~\ref{thm:sufficient} may be constructed as follows.
    Let $\left(U_B\right)_{:,1:\abs{J}}=\brac*{V_A}_J$. Then, for example, let
    \[\left(U_B\right)_{:,(\abs{J}+1):r_B)}=\left(V_A\right)_{:,(r_A+1):(r_A+(r_B-\abs{J}))},\]
    that is attach to the matrix $\left(U_B\right)_{:,1:\abs{J}}$ sufficiently many orthonormal columns from $\nullsp{A}$ so it has $r_B$ columns.
    Find any orthonormal basis of the orthogonal complement of the range of the matrix  $\left(U_B\right)_{:,1:r_B}$ thus obtained and put it in columns of $\left(U_B\right)_{:,(r_B+1):n}$. Arbitrarily assign the first $r_B$ diagonal elements of $\Sigma_B$ with positive numbers and take any for $V_B$ take any unitary matrix in $V_B\in\bK^{k\times k}$. See Section~\ref{sec:matlab} for a concrete implementation in MATLAB code.

    In particular, any matrix $B$ of rank $r_B$ not greater than $n$ can be constructed by choosing for the right singular vectors of $B$, for example, up to $r_A$ singular vectors of $A$ and up to $n-r_A$ vectors from the orthonormal basis $v_{A,r_A+1},\ldots,v_{A,n}\in\nullsp{A}$ . For arbitrary $k\ge r_B$, one may take any $\Sigma_B=\diag(\sigma_{B,1},\ldots,\sigma_{B,r_b},0,\ldots,0)\in\bK^{n\times k}$ generalized diagonal matrix with $\sigma_{B,i}>0$ and any unitary matrix $V_B\in\bK^{k\times k}$.
\end{remark}

In fact, the example from the Introduction was obtained in this way. As an immediate consequence of above remark we have.

\begin{corollary}
    For any matrix $A\in\bK^{m\times n}$ and any $s\in\{1,\ldots,n\}$ and $k\ge s$ there exists a matrix $B\in\bK^{n\times k}$ such that
    \[\rank B=s\quad\text{and}\quad \pinv{(AB)}=\pinv{B}\pinv{A}.\]
\end{corollary}

A symmetric statement to Theorem~\ref{thm:sufficient}, for $J\subset\{1,\ldots,r_B\}$, will be obtained as a consequence by the following trick.

\begin{lemma}\label{lem:pinv_and_conj}
    For any matrix $A\in\bK^{m\times n}$
    \[\pinv{\left(A^*\right)}=\left(\pinv{A}\right)^*,\]
    and hence
    \[\pinv{(AB)}=\pinv{B}\pinv{A}\Longleftrightarrow \pinv{(B^*A^*)}=\pinv{\left(A^*\right)}\pinv{\left(B^*\right)}.\]
\end{lemma}

\begin{proof}
    The first statement follows from the formula expressing the pseudoinverse in terms of any SVD decomposition and the latter follows from the first one.
\end{proof}

\begin{corollary}\label{cor:conjg_ROL}
     Let $A\in\bK^{m\times n},\ B\in\bK^{n\times k}$ be two matrices of ranks $r_A, r_B$ with  given SVD decompositions $A=U_A\Sigma_A V_A^*$ and $B=U_B\Sigma_B V_B^*$. 
 If there exists a (possibly empty) set $J\subset\{1,\ldots,r_B\}$, indexing first $r_B$ columns of matrix $U_B$ (i.e., the left singular vectors of $B$), such that
\[\linsp \left(v_i\in\bK^n \mid v_i\notin\nullsp{B^*},\ i\in\{1,\ldots,r_A\}\right)=\rng{\left(U_B\right)_J},\]
with $v_i$ being the $i$-th column of matrix $V_A$, then
    \[\pinv{(AB)}=\pinv{B}\pinv{A}.\]

\end{corollary}

\begin{proof}
    Follows from Theorem~\ref{thm:sufficient} and Lemma~\ref{lem:pinv_and_conj}.
\end{proof}

\begin{remark}
    When $\bK=\bR$ Corollary~\ref{cor:conjg_ROL} can be restated in terms of rows of matrix $A$.
\end{remark}

The sufficient condition can be easily proven under more restrictive assumption.

\begin{lemma}
    Let $A=U_A\Sigma_A V_A$ be any SVD decomposition of matrix $A\in\bK^{m\times n}$. If $B\in\bK^{n\times k}$ and there exists a non--empty set $J\subset \{1,\ldots,r_A\}$, where $r_A=\rank A$ such that 
   \begin{equation}\label{eq:restr_span}
   \rng{B}=\rng{\left(V_A\right)_J},    
   \end{equation}
    then $\pinv{(AB)}=\pinv{B}\pinv{A}$.

    If there exists a non--empty set $J\subset \{1,\ldots,n\}$,such that equation~(\ref{eq:restr_span}) holds then $\pinv{B}\pinv{A}\in \minv{(AB)}{1,2,3}$, i.e., $\pinv{B}\pinv{A}$ satisfies conditions~(\ref{eq:Penrose_cond1}),(\ref{eq:Penrose_cond2}),(\ref{eq:Penrose_cond3}) for matrix $AB$.
\end{lemma}

\begin{proof}
   For simplicity assume that $J=\{1,\ldots,k\}$ where $k\le r_A$. Then, by definition
   $\rng{B}\subset\rng{A^*}$ and therefore $\rng{BB^*A^*}\subset\rng {A^*}$.
   Let $v\in\rng{B}$. There exist $\alpha_i\in\bK$ such that $v=\alpha_1 v_{A,1}+\ldots+\alpha_k v_{A,k}$. By construction, c.f. Theorem~\ref{thm:construction_of_SVD}, vectors $v_{A,i}$ are eigenvectors of $A^*A$. Therefore
   \[A^*Av=\sigma_{A,1}^2\alpha_1 v_{A,1}+\ldots+\sigma_{A,k}^2\alpha_k v_{A,k}\in\rng{B}.\]
   Therefore $\rng{A^*AB}\subset\rng{B}$ and both Greville's conditions are satisfied.

   When $J\subset\{1,\ldots,n\}$ some of vectors $v_i$ might be in the null--space of $A$ and the condition $\rng{A^*AB}\subset\rng{B}$ holds, which by~\cite[Theorem~2.4(e)]{Tian_rank_formulas} is equivalent to $\pinv{B}\pinv{A}\in\minv{(AB)}{1,2,3}$.
\end{proof}

\begin{example}
Let $A=\begin{bmatrix}1 & 0 \\ 0 & 0\end{bmatrix}$
 be a matrix with given SVD decomposition $A=IAI$. The for 
 $B=\begin{bmatrix}1 & 1 & 0  \\ 0 & 1 & 1\end{bmatrix}$ we have $\rng{B}=\rng{V_A}=\bK^2$ but
 \[\pinv{(AB)}=\frac{1}{2}\left[\begin{array}{cc} 1 & 0\\ 1 & 0\\ 0 & 0 \end{array}\right]\neq \frac{1}{3}\left[\begin{array}{rr} 2 & 0\\ 1 & 0\\ -1 & 0 \end{array}\right]=\pinv{B}\pinv{A}.\]
 \end{example}

%% file: SVD.tex
The unitary group $U(n)$ acts transitively and freely on the right on the set of all orthonormal bases of an $n$-dimensional subspace. In concrete terms.

\begin{lemma}
    Let $U,V\in\bK^{m\times n}$ be two matrices with orthonormal columns, that is $U^*U=V^*V=I$. Then
    \[\rng{U}=\rng{V}\Longleftrightarrow\text{there exists a unique $Q\in\bK^{n\times n}$ such that $Q^*Q=I$ and $V=UQ$}.\]
\end{lemma}

\begin{proof}
    $(\Longrightarrow)$ there exists an invertible matrix $Q\in\bK^{n\times n}$ such that $V=UQ$ and therefore $Q=U^*V$
    \[Q^*Q=V^*UU^*V=V^*\brac*{UU^*}V=V^*V=I,\]
    as $UU^*$ is an orthogonal projection onto $\rng{V}=\rng{U}$.\\
    $(\Longleftarrow)$ obvious.
\end{proof}

Contrary to popular fallacy, the singular value decomposition (SVD) of matrix $A$ is not unique, even if $A$ is generic. A careful analysis of the proof shows that there are many choices to be made. The well--known proof of SVD is included below for an easy reference only.

\begin{theorem}\label{thm:construction_of_SVD}
    For any matrix $A\in\bK^{m\times n}$ there exist unitary matrices $U\in\bK^{m\times m},V\in\bK^{n\times n}$ and a generalized diagonal matrix $\Sigma\in\bK^{m\times n}=\diag(\sigma_1,\ldots,\sigma_r,0,\ldots,0)$ with $\sigma_1\ge \sigma_2\ge \ldots\ge \sigma_r>0$ such that
    \[A=U\Sigma V^*.\]
\end{theorem}

\begin{proof}
     Matrix $A^* A\in \bK^{n\times n}$ is Hermitian hence diagonalizable by the Spectral Theorem. It is also positive semidefinite with real positive eigenvalues $\lambda_1 > \lambda_2 > \ldots >\lambda_p> 0$ and possibly an eigenvalue equal to $0$. Eigenspaces $V_{(\lambda_i)}=\{v\in\bK^n\mid A^*Av=\lambda_iv\}$ are pairwise orthogonal and give the orthogonal direct sum decomposition
     \[\bK^n=V_{(\lambda_1)}\oplus\ldots\oplus V_{(\lambda_{p})}\oplus V_{(0)},\]
     where  $\dim V_{(0)}=n-r$.
     Let $v_1^{(i)},\ldots,v_{m_i}^{(i)}\in\bK^n$ be an orthonormal basis of $V_{(i)}$ where $m_i=\dim V_{(i)}, i=1,\ldots,p$. For fixed i, orthonormal bases of $V_{(i)}$ are parametrized by points of $O(m_i,\bK)$.
     Set $n_1=0$ and $n_i=m_1+\ldots+m_{i-1}$ for $i\ge 2$ and let
     \[v_{n_i+j}=v_j^{(i)},\quad \sigma_{n_i+j}=\sqrt{\lambda_i},\]
     for $i=1,\ldots,p$ and $j=1,\ldots,m_i$. Finally $v_1,\ldots,v_r$ is an orthonormal basis of $V_{(\lambda_1)}\oplus\ldots\oplus V_{(\lambda_{p})}$. Let
     \[u_i=\frac{1}{\sigma_i}Av_i\in\bK^m,\]
     for $i=1,\ldots,r$.
Then
\[u_i^* u_j=\frac{1}{\sigma_i\sigma_j}v_i^* A^* Av_j=\frac{\lambda_j}{\sigma_i\sigma_j}v_i^*v_j=\left\{
  \begin{matrix}
    0 & i\neq j,\\
    1 & i=j,
  \end{matrix}
  \right. ,\quad\text{for}\quad i,j=1,\ldots,r.\]
  Let $v_{r+1},\ldots,v_n\in\bK^n$ be any orthonormal basis of $V_{(0)}$
  and let $u_{r+1},\ldots,u_m\in\bK^m$ be any orthonormal basis of the orthogonal complement of $\linsp(u_1,\ldots,u_r)$. Those choices are given by the point in $U(n-r,\bK)$ and in $U(m-r,\bK)$, respectively. Finally, let $U\in\bK^{m\times m}$ be an orthogonal matrix with columns $u_i$, let $V\in\bK^{n\times n}$ be an orthogonal matrix with columns $v_i$, and let $\Sigma=\diag(\sigma_1,\ldots,\sigma_r,0,\ldots,0)\in\bK^{m\times r}$. Then, by construction, the matrices $U,V$ are unitary and
  \[A=U\Sigma V^*.\]
  
\end{proof}

\begin{corollary}
    If
    \[A=U\Sigma V^*=\widebar{U}\widebar{\Sigma} \widebar{V}^*,\]
    are two SVD decompositions then $\Sigma=\widebar{\Sigma}$ and there exist matrices 
    \[Q_i\in U(m_i,\bK)\quad\text{for}\quad i=1,\ldots,p,\quad Q_{\cN}\in U(n-r,\bK),\quad Q_{\cN^*}\in U(m-r,\bK),\]
    such that
    \[V=\widebar{V}\diag(Q_1,\ldots,Q_p,Q_{\cN}),\]
    \[U_{:,1:r}=AV_{:,1:r}\left(\diag(\sigma_1,\ldots,\sigma_r)\right)^{-1},\quad 
    \widebar{U}_{:,1:r}=A\widebar{V}_{:,1:r}\left(\diag(\sigma_1,\ldots,\sigma_r)\right)^{-1},\]
    \[U_{:,(r+1):n}=\widebar{U}_{:,(r+1):n}Q_{\cN^*}.\]
    where $\diag$ denotes a block diagonal matrix.
\end{corollary}

\begin{proof}
  Under the assumptions
  \[A^*A=V\diag(\sigma_1^2,\ldots,\sigma_r^2,0,\ldots,0)V^*=\widebar{V}\diag(\widebar{\sigma}_1^2,\ldots,\widebar{\sigma}_r^2,0,\ldots,0)\widebar{V}^*.\]
  By the uniqueness of the eigenvalues of $A^*A$ and positivity of $\sigma_i$'s
  \[\sigma_i=\widebar{\sigma}_i\quad\text{for}\quad i=1,\ldots,r.\]
  Therefore, orthonormal bases of the eigenspaces $V_{(\lambda_i)}$ differ by a unitary change-of-coordinates matrix $Q_i$ for $i=1,\ldots,p$. Similarly, $v_{r+1},\ldots,v_n$ and $\widebar{v}_{r+1},\ldots,\widebar{v}_n$ form an orthogonal basis of $V_{(0)}=\nullsp{A}$.
  By construction, the vectors $u_i$ and $\widehat{u_i}$ for $i=1,\ldots,r$ are uniquely determined by the vectors $v_i$ and $\widehat{v_i}$, and for $i>r$ vectors $u_i$ and $\widehat{u_i}$ span $\nullsp{A^*}$.
\end{proof}

\begin{corollary}
    SVDs of matrix $A$ are parametrized by the $\bK$-points of the group
    \[U(m_1,\bK)\times \ldots\times U(m_p,\bK)\times U(n-r,\bK)\times U(m-r,\bK).\]
\end{corollary}

\begin{remark}\label{rem:bases_of_im_ker}
    Vectors $v_1,\ldots,v_r$ (i.e., the right singular vectors) form an orthonormal  basis of $\rng{A^*}$, vectors $v_{r+1},\ldots,v_n$ form an orthonormal  basis of $\nullsp{A}$. Vectors $u_1,\ldots,u_r$ (i.e., the left singular vectors) form an orthonormal  basis of $\rng{A}$, vectors $u_{r+1},\ldots,u_m$ form an orthonormal  basis of $\nullsp{A^*}$. This follows from the following facts
    \[\rng{A^*A}=\rng{A^*},\quad \nullsp{A^*A}=\nullsp{A},\]
    \[\rng{AA^*}=\rng{A},\quad \nullsp{AA^*}=\nullsp{A^*},\]
    \[A^*A=V\Sigma^*\Sigma V^*,\quad AA^*=U\Sigma\Sigma^* U^*.\]
\end{remark}

\begin{remark}
There is an alternative construction of the SVD decomposition which starts from $AA^*$ instead of $A^*A$. It can be seen as the SVD decomposition of $A^*$. 
\end{remark}

%% file: necessary.tex
This part uses standard results from linear algebra concerning an endomorphism of a finite--dimensional vector space over $\bK$ with an inner product, applied to $\bK^n$ with the standard scalar product. In this setting, a self--adjoint operator corresponds to a Hermitian matrix.

\begin{lemma}[Spectral Theorem]
Hermitian matrix is diagonalizable in an orthonormal basis with real eigenvalues. The eigenspaces of a Hermitian matrix corresponding to different eigenvalues are orthogonal. The orthogonal complement of an invariant subspace of a Hermitian matrix is also invariant. Restriction of a diagonalizable endomorphism to an invariant subspace is diagonalizable. 
\end{lemma}

\begin{proof}
 A proof of the Spectral Theorem can be found in ~\cite[Chapter~XV, \S 6--\S 7]{LangAlgebra}. Invariance of the orthogonal complement is an easy exercise.
 The last claim follows as an endomorphism is diagonalizable if and only if its minimal polynomial is semi--simple (i.e., has simple roots), c.f.~\cite[Chapter~XIV, Ex.~13]{LangAlgebra}, and minimal polynomial of an endomorphism is divisible by a minimal polynomial of its restriction to an invariant subspace.
\end{proof}

\begin{theorem}\label{thm:necessary}
Let $A\in\bK^{m\times n},B\in\bK^{n\times k}$ be two non--zero matrices such that
 $\pinv{(AB)}=\pinv{B}\pinv{A}$. Then there exist SVD decompositions $A=U_A\Sigma_A V_A^*$ and $B=U_B\Sigma_B V_B^*$ such that  $\rng{B}$ is spanned by some columns of  matrix $V_A$, $\rng{A^*}$ is spanned by some columns of matrix $U_B$ and
\begin{equation}\label{eq:spaces}
\linsp \left(u_{B,i}\in\bK^n \mid u_{B,i}\notin\nullsp{A},\ i\in\{1,\ldots,r_B\}\right)=
\linsp \left(v_{A,i}\in\bK^n \mid v_{A,i}\notin\nullsp{B^*},\ i\in\{1,\ldots,r_A\}\right). 
\end{equation}

\end{theorem}
\begin{proof}
    Let $r_A=\rank A,r_B=\rank B$. The $i$-th column of the matrix $U_A$ will be denoted by $u_{A,i}$, etc. Equation $\pinv{(AB)}=\pinv{B}\pinv{A}$ holds if and only if both Greville's conditions are satisfied, cf.~\cite[Ex.~22,p.~160]{BenGreville}.
    The first Greville's condition
    \[\rng{A^*AB}\subset \rng{B}\]
    implies that $\rng{B}$ is an invariant subspace of Hermitian matrix $A^*A$. Since $\bK^n=\rng{B}\oplus \nullsp{B^*}$ and any invariant subspace of a Hermitian matrix has an invariant orthogonal complement, it follows that $\nullsp{B^*}$ is also invariant for $A^*A$. The restriction of a diagonalizable operator to an invariant subspace is diagonalizable, therefore $A^*A$ diagonalizes  when restricted to $\rng{B}$ and to $\nullsp{B^*}$. Therefore, there exists an orthonormal eigenbasis $w_{B,1},\ldots,w_{B,r_B}$ of $\rng{B}$ and an orthonormal eigenbasis $z_{B,1},\ldots,z_{B,k-r_B}$ of $\nullsp{B^*}$ for matrix $A^*A$. The eigenvectors in both bases, corresponding to the non--zero eigenvalues of $A^*A$ can be taken as the right singular vectors for $A$, i.e, $v_{A,1},\ldots,v_{A,r_A}$.
    That is,
    \[\{v_{A,1},\ldots,v_{A,r_A}\}=\{w_{B,j}\in\rng{B}\mid A^*Aw_{B,j}\neq 0\}\cup 
    \{z_{B,j}\in\nullsp{B^*}\mid A^*Az_{B,j}\neq 0\}.\]
    By construction, vectors $v_{A,1},\ldots,v_{A,r_A}$ form an orthonormal basis of $\rng{A^*}$, and 
    \[v_{A,i}\in\rng{B}\quad\text{or}\quad v_{A,i}\in\nullsp{B^*},\]
    for $i=1,\ldots,r_A$. The left singular vectors of $A$, i.e. $u_{A,1},\ldots,u_{A,r_A}$, an orthonormal basis of $\rng{A}$, are uniquely determined by the right ones, cf. Theorem~\ref{thm:construction_of_SVD}. Vectors corresponding to $\nullsp{A}$, that is, $v_{A,r_A+1},\ldots,v_{A,n}$, and $\nullsp{A^*}$, that is $u_{A,r_A+1},\ldots,u_{A,m}$, can be chosen arbitrarily (in the case of $\nullsp{A}$, for example, by choosing eigenvectors of $A^*A$ corresponding to zero eigenvalues, that is $w_{B,i}$'s and $z_{B,i}$'s such that, respectively, $A^*Aw_{B,i}=0$ and $A^*Az_{B,i}=0$, cf. Theorem~\ref{thm:construction_of_SVD}).

    By an analogous construction applied to the second Greville's condition
    \[\rng{BB^*A^*}\subset \rng{A^*},\]
    there exists an orthonormal eigenbases of $\rng{A^*}$ and of $\nullsp{A}$ for the matrix $BB^*$. Eigenvectors in both bases, corresponding to the non--zero eigenvalues of $BB^*$ can be taken as the right singular vectors for $B^*$, or equivalently, the left singular vectors of $B$, that is $u_{B,1},\ldots,u_{B,r_B}$. By construction, vectors $u_{B,1},\ldots,u_{B,r_B}$ form an orthonormal basis of $\rng{B}$, and 
    \[u_{B,i}\in\rng{A^*}\quad\text{or}\quad u_{B,i}\in\nullsp{A},\]
    for $i=1,\ldots,r_B$.

    Consider 
    \[\linsp \left(u_{B,i}\in\bK^n \mid u_{B,i}\notin\nullsp{A},\ i\in\{1,\ldots,r_B\}\right).\]
    For any $u_{B,i}\in\rng{B}$ such that $u_{B,i}\notin \nullsp{A}$ we have $u_{B,i}\in\rng{A^*}$ and there exist $\alpha_{ij}\in\bK$ such that
    \[u_{B,i}=\sum_{j=1}^{r_B}\alpha_{ij}w_{B,j},\]
    since $w_{B,1},\ldots,w_{B_j}$ is a basis of $\rng{B}$. By construction, either $w_{B,j}\in\nullsp{A}$ or $w_{B,j}\in\rng{A^*}$.
    Since $u_{B,i}\in\rng{A^*}$, by orthogonality
    \[u_{B,i}=\sum_{\substack{j\in\{1,\ldots,r_B\}\text{ such that}\\w_{B,j}\in\rng{A^*}}}\alpha_{ij}w_{B,j}\in
    \linsp \left(v_{A,i}\in\bK^n \mid v_{A,i}\notin\nullsp{B^*},\ i\in\{1,\ldots,r_A\}\right)
    ,\]
    as $w_{B,j}$'s in $\rng{A^*}$ are exactly those $v_{A,i}$'s in $\rng{B}$, orthogonal to $\nullsp{B^*}$.

    The reverse inclusion can be proven in a similar way.
\end{proof}

\begin{remark}\label{rem:range_of_Grevilles_spaces}
    If $\pinv{\brac*{AB}}=\pinv{B}\pinv{A}$ then the space spanned simultaneously by some right singular vectors of $A$ and left singular vectors of $B$ is equal to
    \[\rng{A^*AB}=\rng{BB^*A}=\rng{A^*}\cap\rng{B}.\]
    By definition, it is equal to the image of $A^*A$ restricted to $\rng{B}$ and the image of $BB^*$ restricted to $\rng{A^*}$. Moreover if $v\in\rng{A^*}\cap\rng{B}$ then there exist $\alpha_i,\beta_i\in\bK$ such that
    \[v=\sum_{i=1}^{r_A}\alpha_i v_{A,i}=\sum_{i=1}^{r_B}\beta_i u_{B,i},\]
    but coefficients at all components either in $\nullsp{A}$ or $\nullsp{B^*}$ vanish.
\end{remark}

\begin{corollary}\label{cor:ROL_AB_implies_possible_ROLs}
    Assume that $\pinv{\brac*{AB}}=\pinv{B}\pinv{A}$. Let 
    \[J_A=\{i\in\{1,\ldots,r_A\}\mid v_{A,i}\notin\nullsp{B^*}\},\]
    \[J_B=\{i\in\{1,\ldots,r_B\}\mid u_{B,i}\notin\nullsp{A}\}.\]
    To simplify the notation let
    \[U_{J,A}=\brac*{U_A}_{J_A},\quad V_{J,A}=\brac*{V_A}_{J_A},\]
    \[U_{J,B}=\brac*{U_B}_{J_B},\quad V_{J,B}=\brac*{V_B}_{J_B},\]
    \[\widetilde{\Sigma}_A=\brac*{\Sigma_A}_{(J_A,J_A)},\quad \widetilde{\Sigma}_B=\brac*{\Sigma_B}_{(J_B,J_B)}.\]
    Then
    \[\rng{AB}=\rng{U_{J_A}},\quad \rng{B^*A^*}=\rng{V_{J_B}},\]
    and
    \[\pinv{\brac*{AB}}=\brac*{V_{J_B}\widetilde{\Sigma}_B^{-1}U_{J_B}^*}
    \brac*{V_{J_A}\widetilde{\Sigma}_A^{-1}U_{J_A}^*}.\]
    Moreover, the reverse order law holds for the following cases
    \[\pinv{\brac*{\brac*{AB}B^*}}=\pinv{\brac*{B^*}}\pinv{\brac*{AB}},\quad \pinv{\brac*{\brac*{AB}\pinv{B}}}=B\pinv{\brac*{AB}},\]
    \[\pinv{\brac*{A^*\brac*{AB}}}=\pinv{\brac*{AB}}\pinv{\brac*{A^*}},\quad \pinv{\brac*{\pinv{A}\brac*{AB}}}=\pinv{\brac*{AB}}A,\]
     \[\pinv{\brac*{\brac*{A^*A}B}}=\pinv{B}\pinv{\brac*{A^*A}},\quad 
    \pinv{\brac*{A\brac*{B B^*}}}=\pinv{\brac*{BB^*}}\pinv{A},\]
    \[\pinv{\brac*{\brac*{\pinv{A}A}B}}=\pinv{B}\brac*{\pinv{A}A},\quad 
    \pinv{\brac*{A\brac*{B\pinv{B}}}}=\brac*{B\pinv{B}}\pinv{A}.\]
    and
    \[\pinv{\brac*{\pinv{B}A^*}}=\pinv{\brac*{A^*}}B,\quad \pinv{\brac*{B^*\pinv{A}}}=A\pinv{\brac*{B^*}}.\]
    
\end{corollary}

\begin{proof}
    To see the first identity, use the sufficient condition and
    \[\rng{\brac*{AB}^*\brac*{AB}B^*}=\rng{B^*A^*ABB^*}=\rng{V_{J_B}},\]
    \[\rng{B^*\brac*{B^*}^*\brac*{AB}^*}=\rng{B^*BB^*A^*}=\rng{V_{J_B}}.\]
    The other identities follow in a similar manner.
    To see that formula for $\pinv{\brac*{AB}}$ let $Q=U_{J_B}^*V_{J_A}$. Then $Q$ is an orthonormal matrix and 
    \[AB=    U_{J_A}\widetilde{\Sigma}_A Q\widetilde{\Sigma}_B V_{J_B}^*.\]
    This gives a rank decomposition of matrix $AB$.
\end{proof}

\begin{corollary}\label{cor:image_of_AB_under_1234}
    If $\pinv{\brac*{AB}}=\pinv{B}\pinv{A}$ then
    \[AB\pinv{\brac*{AB}}=\brac*{AA^*AB}\pinv{\brac*{AA^*AB}}=
    \brac*{ABB^*A^*}\pinv{\brac*{ABB^*A^*}}.\]
\end{corollary}

\begin{proof}
    The image of $AB$ is $A\brac*{\rng{A^*}\cap\rng{B}}$, on the other hand $\rng{A^*}\cap\rng{B}=\rng{A^*AB}=\rng{BB^*A^*}$.
\end{proof}

\begin{remark}
    Equation~(\ref{eq:spaces}) holds for {\bf some} SVD decompositions of matrices $A,B$ which does not imply that it will be true for {\bf any} SVD decompositions. See Example~\ref{ex:not_for_any}.
\end{remark}

%% file: commuting.tex
In this section, several new equivalent conditions to $\pinv{\brac*{AB}}=\pinv{B}\pinv{A}$ are given. Through the section, it assumed that $A\in\bK^{m\times n},B\in\bK^{n\times k}$.

\begin{theorem}\label{thm:pAA_BB_and_AA_BpB_commute}
    \[\pinv{\brac*{AB}}=\pinv{B}\pinv{A} \Longleftrightarrow \pinv{A}ABB^*=BB^*\pinv{A}A\text{ and }B\pinv{B}A^*A=A^*AB\pinv{B},\]
    that is $\pinv{A}A$ commutes with $BB^*$ and $B\pinv{B}$ commutes with $A^*A$.
\end{theorem}

\begin{proof}
   $(\Longleftarrow)$ Both Greville's conditions are satisfied as $\rng{A^*}=\pinv{A}A, \rng{B}=B\pinv{B}$
    \[\rng{A^*AB}\subset \rng{B}\Longleftrightarrow \rng{A^*AB\pinv{B}}\subset \rng{B\pinv{B}},\]
    \[\rng{BB^*A^*}\subset \rng{A^*}\Longleftrightarrow \rng{BB^*\pinv{A}A}\subset \rng{\pinv{A}A}.\]
    $(\Longrightarrow)$\\
    Since $\rng{BB^*A^*}\subset \rng{A^*}$
    \[\bK^n=\rng{A^*}\oplus \nullsp{A}=\brac*{\rng{A^*}_*\oplus\rng{A^*}_0}\oplus\brac*{\nullsp{A}_*\oplus\nullsp{A}_0}, \]
    where $\rng{A^*}_0,\nullsp{A}_0\subset\nullsp{B^*}=\nullsp{BB^*}$ and $\rng{A^*}_*,\nullsp{A}_*$ are direct sums of eigenspaces
    of $BB^*$ for non--zero eigenvalues.
    Let $v\in\bK^n$ be any vector where
    \[v=v_{\rng{A^*},*}+v_{\rng{A^*},0}+v_{\nullsp{A},*}+v_{\nullsp{A},0},\]
    is the decomposition of $v$ according to this direct sum.
    Then
    \[BB^*v=BB^*v_{\rng{A^*},*}+BB^*v_{\nullsp{A},*},\]
    \[\pinv{A}ABB^*v=BB^*v_{\rng{A^*},*}.\]
    On the other hand
    \[\pinv{A}Av=v_{\rng{A^*},*}+v_{\rng{A^*},0},\]
    \[BB^*\pinv{A}Av=BB^*v_{\rng{A^*},*}=\pinv{A}ABB^*v. \]
    
\end{proof}

\begin{lemma}\label{lem:ROL_AB_implies_ROL_AA_BB}
\[\pinv{\brac*{AB}}=\pinv{B}\pinv{A}\Longleftrightarrow \pinv{\brac*{\brac*{A^*A}\brac*{BB^*}}}=\pinv{\brac*{BB^*}}\pinv{\brac*{A^*A}}.\]    
\end{lemma}

\begin{proof}
    The proof in both directions follows from the Greville's conditions~\ref{thm:greville_conds}. \\
    $(\Longrightarrow)$ If $\rng{B}$ is an invariant subspace for $A^*A$ then $\rng{BB^*}=\rng{B}$ is an invariant subspace of the matrix $\brac*{A^*A}^2$. Similarly for $\rng{A^*}$. \\
    $(\Longleftarrow)$ If $\rng{BB^*}=\rng{B}$ is an invariant subspace of the diagonalizable Hermitian matrix $\brac*{A^*A}^2$ hence it is spanned by eigenvectors of $\brac*{A^*A}^2$ which are also eigenvectors of $A^*A$ (diagonalizable Hermitian matrices $A^*A$ and $\brac*{A^*A}^2$ commute, so they are simultaneously diagonalizable). Similarly for $\rng{A^*}$.
    
\end{proof}

\begin{lemma}\label{lem:orthogonal_proj}
    \[\pinv{\brac*{AB}}=\pinv{B}\pinv{A}\Longleftrightarrow  P=B\pinv{\brac*{AB}}A\text{ is an orthogonal projection}.\]    
\end{lemma}

\begin{proof}
  By~\cite[Theorem~3.1]{Cerny}
  \[\rng{P}=\rng{BB^*A},\quad \nullsp{P}=\nullsp{B^*A^*A},\]
  and they are orthogonal exactly when $\rng{BB^*A}=\rng{A^*AB}$ which is equivalent to the reverse order law by 
  Remark~\ref{rem:range_of_Grevilles_spaces} and the Greville's conditions.
\end{proof}

%% file: misc.tex
In this Section we give some concrete examples and discuss the previously--known cases when $\pinv{(AB)}=\pinv{B}\pinv{A}$.

\begin{example}
Let
    \[A=\left[\begin{array}{rrr} 1 & 0 & 1\\ 0 & -1 & 0 \end{array}\right],\quad B=\left[\begin{array}{r} 1\\ 0\\ -1 \end{array}\right].\]
    Then $\pinv{(AB)}=\pinv{B}\pinv{A}=0$ and there exists SVD decomposition $A=U_A\Sigma_A V_A^*$ where
    \[U_A=\left[\begin{array}{rr} 1 & 0\\ 0 & -1 \end{array}\right],\quad \Sigma_A=\left[\begin{array}{ccc} \sqrt{2} & 0 & 0\\ 0 & 1 & 0 \end{array}\right],V_A=\frac{1}{\sqrt{2}}\left[\begin{array}{rcr} 1 & 0 & 1\\ 0 & \sqrt{2} & 0\\ 1 & 0 & -1 \end{array}\right],\]
    There exists SVD decomposition $B=U_B\Sigma_B V_B^*$ where
    \[U_B=\frac{1}{\sqrt{2}}\left[\begin{array}{rcr} 1 & 0 & 1\\ 0 & \sqrt{2} & 0\\ -1 & 0 & 1 \end{array}\right],\quad
    \Sigma_B=\left[\begin{array}{c} \sqrt{2}\\ 0\\ 0 \end{array}\right],\quad
    V_B=\left[\begin{array}{c} 1 \end{array}\right].\]
    Then

    \[\linsp \left(u_{B,i}\in\bK^n \mid u_{B,i}\notin\nullsp{A},\ i\in\{1,\ldots,r_B\}\right)=
     \linsp \left(v_{A,i}\in\bK^n \mid v_{A,i}\notin\nullsp{B^*},\ i\in\{1,\ldots,r_A\}\right)=\{0\},\]
     where $r_A=2,r_B=1$.
\end{example}

\begin{example}
    Let $A=ST$ be a rank decomposition of matrix $A$ where $S\in\bK^{m\times r},\ T\in\bK^{r\times n}$, i.e. $\rank S=\rank T=r$. Then
    \[\pinv{(ST)}=\pinv{T}\pinv{S}.\]
    Since $S^*,\ T$ are of full row rank $\rng{S^*}=\rng{T}=\bK^r$ and hence
    $\nullsp{S}=\nullsp{T^*}=\{0\}$ and consequently for any SVD decompositions of $S$ and $T$
    \[\linsp \left(u_{T,i}\in\bK^r \mid u_{T,i}\notin\nullsp{S},\ i\in\{1,\ldots,r\}\right)=
     \linsp \left(v_{S,i}\in\bK^r \mid v_{S,i}\notin\nullsp{T^*},\ i\in\{1,\ldots,r\}\right)=\bK^r.\]
\end{example}

\begin{example}
    Assume that $\rng{A^*}=\rng{B}$. It is well--known then that
     \[\pinv{(AB)}=\pinv{B}\pinv{A}.\]
     Since it is impossible for any non--zero vector in $\rng{B}$ to be in $\nullsp{A}$ and for it is impossible for any non--zero vector in $\rng{A^*}$ to be in $\nullsp{B^*}$, therefore, for any SVD decompositions of $A,B$
     \[\linsp \left(u_{B,i}\in\bK^n \mid u_{B,i}\notin\nullsp{A},\ i\in\{1,\ldots,r_B\}\right)=
     \linsp \left(v_{A,i}\in\bK^n \mid v_{A,i}\notin\nullsp{B^*},\ i\in\{1,\ldots,r_A\}\right)=\]
     \[=\rng{A^*}=\rng{B}.\]
     \end{example}
\begin{example}\label{ex:not_for_any}
    Let $A\in\bK^{m\times n}$ be any non--zero matrix such that $\nullsp{A}\neq 0$.  Let $B=Q\in\bK^{n\times n}$ be an orthogonal matrix. Then
     \[\pinv{(AQ)}=\pinv{Q}\pinv{A}=Q^*\pinv{A}.\]
     The right hand side of the formula~(\ref{eq:spaces}) is equal to
 \[\linsp \left(v_{A,i}\in\bK^n \mid v_{A,i}\notin\nullsp{Q^*},\ i\in\{1,\ldots,r_A\}\right)=\rng{A^*}.\]
Any orthonormal matrix $U_B\in\bK^{n\times n}$ induces an SVD decomposition of matrix $Q$ by the formula
 \[Q=U_BI(Q^*U_B)^*,\]
 where $\Sigma_B=I$ and any SVD decomposition of $Q$ is of that form. If none of the vectors $u_{B,1},\ldots,u_{B,n}\in\bK^n$  are in $\nullsp{A}\subsetneqq \bK^n$ then the left hand side of the formula~(\ref{eq:spaces}) is equal to
 \[\linsp \left(u_{B,i}\in\bK^n \mid u_{B,i}\notin\nullsp{A},\ i\in\{1,\ldots,r_B\}\right)=\bK^n\neq \rng{A^*}.\]
 On the other hand, if some columns of $U_B$ form an orthonormal basis of $\rng{A^*}$ and the remaining ones an orthonormal basis of $\nullsp{A}$ then
 \[\linsp \left(u_{B,i}\in\bK^n \mid u_{B,i}\notin\nullsp{A},\ i\in\{1,\ldots,r_B\}\right)=\rng{A^*}.\]
 This shows that condition~\ref{eq:spaces} does not hold for any SVD decompositions of $A$ and $B$.
\end{example}

\begin{remark}
Theorems~\ref{thm:sufficient} and~\ref{thm:necessary} show that when the reverse order law holds, it is possible to choose appropriate vectors $u_i$'s and $v_i$ such the equations~(\ref{eq:GA}) and (\ref{eq:GB}), as linear equations in unknowns $\sigma_{A,i}^2$ and $\sigma_{B,i}^2$ have zero coefficients. Therefore, when the reverse order law holds, for suitable matrices $V_A$ and $U_B$, as in Theorem~\ref{thm:necessary} (with compatible matrices $U_B,V_A$), the reverse order law holds too for matrices $U_A\widebar{\Sigma}_A,V_A,U_B\widebar{\Sigma}_B V_B$ where $\widebar{\Sigma}_A,\widebar{\Sigma}_B$ are arbitrary diagonal matrices of ranks $r_A,r_B$, respectively.
\end{remark}

%% file: geometric.tex
In this section we give a equivalent geometric condition for $\pinv{(AB)}=\pinv{B}\pinv{A}$ in terms of the principal angles between $\rng{A^*}$ and $\rng{B}$, cf. Theorem~\ref{thm:geo_suff}.

\begin{definition}
    Let $\rng{U},\rng{V}\subset\bK^n$ be two subspaces given by matrices with orthonormal columns. The cosines of principal angles between $\rng{U},\rng{V}$ are equal to the singular values (with $0$ included) of the matrix $V^*U$, c.f.~\cite{BjorckGolub},\cite{GolubZha}.
\end{definition}

The principal angles are well--defined, that is, they do not depend on the choice of orthogonal bases of $\rng{U}$ and $\rng{V}$.

\begin{theorem}\label{thm:geo_suff}
    If the principal angles between the subspaces 
    $\rng{A^*}$ and $\rng{B}$ belong to the set $\bracc*{0,\frac{\pi}{2}}$ and $\rng{A^*}\cap\rng{B}$
    is spanned by eigenvectors of $A^*A$ and of $BB^*$ corresponding to the non--zero eigenvalues then $\pinv{(AB)}=\pinv{B}\pinv{A}$.
\end{theorem}

\begin{proof}
    The assumption implies that there exist matrices $V,U$ with orthonormal columns such that
    \[\rng{V}=\rng{A^*},\quad \rng{U}=\rng{B},\]
    and orthogonal matrices $P\in\bK^{r_A\times r_A},Q\in\bK^{r_B\times r_B}$ such that
    \[V^*U=P\left[\begin{array}{rr} I & 0\\ 0 & 0 \end{array}\right]Q^*,\]
    where $I$ is some $s$-by-$s$ unit matrix with $s\le\min(r_A,r_B)$. Therefore
    \[(VP)^*(UQ)=\left[\begin{array}{rr} I & 0\\ 0 & 0 \end{array}\right],\]
    and the first $s$ columns of $VP$ (equivalently of $UQ$) form an orthonormal basis of $\rng{A^*}\cap\rng{B}$. Since columns of $VP$ span $\rng{A^*}$ if follows that $\rng{(UQ)_{:,s:r_B}}\subset\nullsp{A}$, and similarly, since columns of $UQ$ span $\rng{B}$ if follows that $\rng{(VP)_{:,s:r_A}}\subset\nullsp{B^*}$. Let $W=\rng{A^*}\cap\rng{B}$. By assumption $W$ is an invariant space for both $A^*A$ and $BB^*$. Therefore, Greville's conditions
    \[\rng{A^*AB}\subset\rng{B}\quad\text{and}\quad \rng{BB^*A^*}\subset\rng{A^*},\]
    hold as images of columns of $VP$ under $A^*A$ are either $0$ or in $W$ and images of columns of $UQ$ under $BB^*A$ are either $0$ or in $W$.
\end{proof}

The assumption about $\rng{A^*}\cap\rng{B}$ being spanned by the singular vectors is essential, as follows from the following counterexample (it is also obvious, in view of Theorem~\ref{thm:12_eqv}, as being $\{1,2\}$ inverse is not in general equivalent to  being $\{1,2,3,4\}$ inverse.

\begin{example}
Let
\[A=\left[\begin{array}{rrr} 1 & 0 & 1\\ 0 & 1 & -1 \end{array}\right],\quad 
B=\left[\begin{array}{rr} 1 & 0\\ 0 & 1\\ 2 & 3 \end{array}\right].\]
Then the principal angles between $\rng{A^*}$ and $\rng{B}$ are $0,\frac{\pi}{2}$ and
\[\pinv{B}\pinv{A}\in  \minv{(AB)}{1,2}\setminus\bracc*{\minv{(AB)}{3}\cup\minv{(AB)}{4} }\]
    
\end{example}



%% file: SVD_proj.tex
In this section we prove several equivalent conditions for $\pinv{B}\pinv{A}$ being an $\{1,2\}$-inverse of $AB$. Some of them were known before, see~\cite[Theorem~9.1$\langle 49\rangle,\langle 88\rangle,\langle 147\rangle,\langle 161\rangle$]{Tian512} but obtained by different methods.

\begin{lemma}\label{lem:prod_of_proj}
    Let $P,Q\in\bK^{n\times n}$ be matrices of two orthogonal projections, i.e., $P^2=P, Q^2=Q$ and $P^*=P,Q^*=Q$. Then
    \[PQv=v\Longleftrightarrow v\in\rng{P}\cap\rng{Q}.\]
    In particular, if $PQ$ is a projection then $\rng{PQ}=\rng{P}\cap\rng{Q}$.
\end{lemma}

\begin{proof}
    If $v\in\rng{P}\cap\rng{Q}$ then
    \begin{equation}\label{eq:prod_of_proj}
        PQv=Pv=v.    
    \end{equation}
    Assume now $PQv=v$. Then $v\in\rng{P}$. Factor by $P$ the left part of the Equation~(\ref{eq:prod_of_proj}) to get
    \[P(Qv-v)=0.\]
    That is $Qv-v\in \nullsp{P}$ and the null space is orthogonal to $\rng{P}$ therefore
    \begin{equation}\label{eq:perp}
        v^*(Qv-v)=0.    
    \end{equation}
    
    Let $v=x+y$ be the orthogonal decomposition with respect to $\rng{Q}\oplus\nullsp{Q}$. Then
     and $Qv-v=-y$ and the Equation~(\ref{eq:perp}) becomes
    \[(x^*+y^*)(-y)=0\quad\Longrightarrow\quad \norm{y}_2=0,\]
    and therefore $v\in\rng{Q}$.
    
    Finally, if $PQ$ is a projection then $PQv=v$ if and only if $v\in\rng{PQ}$.
\end{proof}

\begin{lemma}\label{lem:12_implies_projections} 
    Let $A\in\bK^{m\times n}$ and $B\in\bK^{n\times k}$ be two matrices. If
    \[\pinv{B}\pinv{A}\in \minv{(AB)}{1,2},\]
    then
    $B\pinv{B}\pinv{A}A$ and $\pinv{A}AB\pinv{B}$ are projections onto $\rng{A^*}\cap\rng{B}$.
\end{lemma}

\begin{proof}
Matrices $B\pinv{B}\pinv{A}A$ and $\pinv{A}AB\pinv{B}$ are idempotent
    \[\brac*{B\pinv{B}\pinv{A}A}\brac*{B\pinv{B}\pinv{A}A}=B\brac*{\pinv{B}\pinv{A}AB\pinv{B}\pinv{A}}A=B\pinv{B}\pinv{A}A,\]
     \[\brac*{\pinv{A}AB\pinv{B}}\brac*{\pinv{A}AB\pinv{B}}=\pinv{A}\brac*{AB\pinv{B}\pinv{A}AB}\pinv{B}=\pinv{A}AB\pinv{B}.\]
     Since $\pinv{A}A$ is the orthogonal projection onto $\rng{A^*}$ and $B\pinv{B}$ is the orthogonal projection onto $\rng{B}$, by Lemma~\ref{lem:prod_of_proj} they both project onto $\rng{A^*}\cap\rng{B}$.
\end{proof}

\begin{lemma}\label{lem:sing_values_1_of_proj}
    Let $P\in\bK^{n\times n}$ be a matrix of a projection, i.e., $P^2=P$. Let $P=U\Sigma V^*$ 
    be any SVD decomposition of $P$ and let $\sigma_1\ge \ldots\ge\sigma_r>0$ denote all singular values of $P$. Then
    \[\sigma_1=\ldots=\sigma_r=1\quad \Longleftrightarrow\quad P=P^*,\]
    i.e., when $P$ is an orthogonal projection.
\end{lemma}

\begin{proof}
    If $P$ is an orthogonal projection then it diagonalizes in an orthonormal basis, i.e.,
    \[P=Q\left[\begin{array}{c|c}
   I  & 0 \\ \hline
   0  & 0
\end{array}\right]Q^*,\]
with eigenvalues $0$ and $1$. By the uniqueness of singular values $\sigma_1=\ldots=\sigma_r=1$.

Let $P=U\Sigma V^*$ be a reduced SVD of $P$. If $\sigma_1=\ldots=\sigma_r=1$ then
\[P^2=UV^*UV^*=P=UV^*.\]
Since $U^*U=I,V^*V=I$, by multiplying the above equation by $U^*$ on the left and by $V$ on the right, we have that $V^*U=I$. Therefore
\[U=V\quad\text{and}\quad P=VV^*\quad\text{is an orthogonal projection}.\]
We have used $v^*u=\norm{v}_2\norm{u}_2\cos \angle(u,v)=1$, if  $\norm{v}_2=\norm{u}_2=1$ then $\cos \angle(u,v)=1$ and $u=v$.
\end{proof}

 \begin{theorem}\label{thm:Greville12}
    \[\pinv{B}\pinv{A}\in \minv{(AB)}{1,2}\Longleftrightarrow \rng{\pinv{A}AB}\subset\rng{B}\text{ and }\rng{B\pinv{B}\pinv{A}}\subset\rng{\pinv{A}}.\] 
 \end{theorem}

 \begin{proof}
     $(\Longleftarrow)$
     Since $\pinv{A}A$ is an orthogonal projection onto $\rng{\pinv{A}}=\rng{A^*}$ and $B\pinv{B}$ is an orthogonal projection onto $\rng{B}$ we have
     \begin{equation}\label{eq:proj2}
         \brac*{B\pinv{B}}\pinv{A}AB=\pinv{A}AB.
     \end{equation}
     \begin{equation}\label{eq:proj1}
         \brac*{\pinv{A}A}B\pinv{B}\pinv{A}=B\pinv{B}\pinv{A},
     \end{equation}
     Multiplying (\ref{eq:proj2}) by $A$ on the left gives
     \[AB\pinv{B}\pinv{A}AB=AB.\]
     Multiplying (\ref{eq:proj1}) by $\pinv{B}$ on the left gives
     \[\pinv{B}\pinv{A}AB\pinv{B}\pinv{A}=\pinv{B}\pinv{A},\]
     i.e., $\pinv{B}\pinv{A}\in \minv{(AB)}{1,2}$.\\
     $(\Longrightarrow)$ Let $P=\pinv{A}AB\pinv{B}$ and $Q=B\pinv{B}\pinv{A}A$.
     By Lemma~\ref{lem:12_implies_projections} matrices $P,Q$ are projections both onto $\rng{A^*}\cap\rng{B}$, that is $P^2=P,Q^2=Q$ and
     \[PQ=Q,\quad QP=P.\]
     Equivalently,
     \[\brac*{\pinv{A}AB\pinv{B}}\brac*{B\pinv{B}\pinv{A}A}=B\pinv{B}\pinv{A}A,\]
     \[\brac*{B\pinv{B}\pinv{A}A}\brac*{\pinv{A}AB\pinv{B}}=\pinv{A}AB\pinv{B},\]
     but those, after canceling the middle terms, imply the inclusions.
 \end{proof}

 \begin{corollary}\label{cor:pAA_BpB_commute}
     \[\rng{\pinv{A}AB}\subset\rng{B}\Longleftrightarrow\rng{B\pinv{B}\pinv{A}}\subset\rng{\pinv{A}}.\]
     In particular, if
     \[\pinv{B}\pinv{A}\in \minv{(AB)}{1,2}\] 
     then $\pinv{A}A$ and $B\pinv{B}$ are simultaneously diagonalizable hence they commute.

 \end{corollary}

 \begin{proof}
     The proof is similar to the proof of Theorem~\ref{thm:necessary}. Assume that
     \[\rng{B\pinv{B}\pinv{A}}\subset\rng{\pinv{A}},\]
     therefore, subspace $\rng{A^*}=\rng{\pinv{A}}$ is an invariant subspace of the projection
     $B\pinv{B}$ and so is its orthogonal complement $\nullsp{A}$. Therefore there exists an orthonormal eigenbasis $v_1,\ldots,v_n$ of $\bK^n=\rng{A^*}\oplus\nullsp{A}$ for  $B\pinv{B}$ with  $v_1,\ldots,v_{r_A}\in\rng{A^*}$ and 
      $v_{r_A+1},\ldots,n\in\nullsp{A}$. Vectors from $\rng{A^*}$ or $\nullsp{A}$ are also eigenvectors for the projection $\pinv{A}A$. Moreover, eigenvectors $v_i$ corresponding to the eigenvalue $1$ of $B\pinv{B}$ form a basis of $\rng{B}$.  To see that
     \[\rng{\pinv{A}AB}\subset\rng{B},\]
     it is enough to show that $\pinv{A}Av_i\in\rng{B}$ for any $v_i\in\rng{B}$. But for any $i=1,\ldots,n$
     \[\pinv{A}Av_i=v_i\text{ or }\pinv{A}Av_i=0.\]
     The converse follows in a similar way.    
 \end{proof}

\begin{theorem}\label{thm:12_eqv}
 Let $A\in\bK^{m\times n}$ and $B\in\bK^{n\times k}$ be two matrices. 
    The following conditions are equivalent
    \begin{enumerate}[i)]
        \item $\pinv{B}\pinv{A}\in \minv{(AB)}{1,2}$,
        \item $AB\in \minv{\brac*{\pinv{B}\pinv{A}}}{1,2}$,
        \item $\rng{\pinv{A}AB}\subset\rng{B}$,
        \item $\rng{B\pinv{B}\pinv{A}}\subset\rng{\pinv{A}}$
        \item the principal angles between $\rng{A^*}$ and $\rng{B}$ belong to the set $\bracc*{0,\frac{\pi}{2}}$,
        \item $\pinv{A}AB\pinv{B}$ is an orthogonal projection,
        \item $B\pinv{B}\pinv{A}A$ is an orthogonal projection,
        \item matrices $\pinv{A}A,B\pinv{B}$ commute, 
        \item matrices $\pinv{A}A,B\pinv{B}$ are simultaneously diagonalizable,
        \item $\pinv{A}AB\pinv{B}=2\pinv{A}A\pinv{\brac*{\pinv{A}A+B\pinv{B}}}B\pinv{B},$
        \item $\pinv{A}AB\pinv{B}\pinv{A}AB\pinv{B}=B\pinv{B}\pinv{A}A$,
        \item eigenvalues of $\pinv{A}AB\pinv{B}\pinv{A}A$ belong to the set $\{0,1\}$.
    \end{enumerate}
\end{theorem}

 \begin{proof}
 Conditions $i)$ and $ii)$ are equivalent by symmetry. By Theorem~\ref{thm:Greville12} and Corollary~\ref{cor:pAA_BpB_commute} condition $i)$ is equivalent to conditions $iii)$ and $iv)$.
 Similarly condition $i)$ implies conditions $viii)$ and $ix)$. Since $\pinv{A}A,B\pinv{B}$ are diagonalizable conditions $viii)$ and $ix)$ are equivalent. Condition $viii)$ implies $vi)$ and $vii)$. Product of two orthogonal projections is an orthogonal projection if and only if they commute, therefore $vi)$ or $vii)$ imply $viii)$.

Let $A=U_A\Sigma_A V_A^*$ and $B=U_B\Sigma_B V_B^*$ be a reduced SVD decompositions of $A$ and $B$. Then the singular values (including $0$) of $\pinv{A}AB\pinv{B}=V_AV_A^*U_BU_B^*$ are the same as of matrix $V_A^*U_BU$.
By Lemma~\ref{lem:sing_values_1_of_proj}, conditions $v)$ and $vi)$ are equivalent.

By \cite[Theorem~8]{AndersonDuffin} the right hand side of $x)$ is an orthogonal projection onto $\rng{A^*}\cap\rng{B}$ in terms of parallel sum. By the last claim of Lemma~\ref{lem:prod_of_proj},
conditions $vi)$ and $x)$ are equivalent.

Condition $vii)$ implies trivially condition $xi)$. Multiplying condition $xi)$ by $B\pinv{B}$ on the left, we can rewrite it as
    \begin{equation}\label{eq:12_equiv_t1}
    \brac*{B\pinv{B}\pinv{A}AB\pinv{B}}\brac*{B\pinv{B}\pinv{A}AB\pinv{B}}=B\pinv{B}\pinv{A}A.
    \end{equation}
    Since multiplying (\ref{eq:12_equiv_t1}) by $B\pinv{B}$ on the right does not change anything on the left hand side, we get that $B\pinv{B}\pinv{A}A=B\pinv{B}\pinv{A}AB\pinv{B}$ is a projection, i.e, condition $vii)$.
 By a similar argument and  Lemma~\ref{lem:sing_values_1_of_proj}, condition $xi)$ implies condition $xii)$.
 Eigenvalues of $\pinv{A}AB\pinv{B}\pinv{A}A$ are squares of singular values of $\pinv{A}AB\pinv{B}$, which in turn are cosines of the prinicipal angles between $\rng{A^*}\cap\rng{B}$, therefore $xii)$ implies $v)$.
 \end{proof}

\begin{corollary}
    If $\pinv{B}\pinv{A}\in \minv{(AB)}{1,2}$ then
    the principal angles $\alpha_i$ between $\rng{A^*}$ and $\rng{B}$ belong to the set $\bracc*{0,\frac{\pi}{2}}$ therefore
    \[\rank AB=\dim(\rng{A^*}\cap\rng{B})=\#\{\alpha_i\mid \alpha_i=0\}.\]
\end{corollary}

\begin{corollary}\label{cor:other_12ROLs}
Since $\rng{A^*}=\rng{A^*A}=\rng{\pinv{A}A}$ and $\rng{B}=\rng{BB^*}=\rng{B\pinv{B}}$
    \[\pinv{B}\pinv{A}\in \minv{(AB)}{1,2}\Longleftrightarrow \pinv{\brac*{BB^*}}\pinv{\brac*{A^*A}}\in \minv{(A^*ABB^*)}{1,2} \Longleftrightarrow \pinv{\brac*{B\pinv{B}}}\pinv{\brac*{\pinv{A}A}}\in \minv{(\pinv{A}AB\pinv{B})}{1,2} .\]
\end{corollary}

The condition with principal angles yields an easy sufficient condition.

 \begin{corollary}
     If $\rng{A^*}\subset\rng{B}$ or $\rng{B}\subset\rng{A^*}$ then
     $\pinv{B}\pinv{A}\in \minv{(AB)}{1,2}$.
 \end{corollary}
This is a counterpart of Corollary~\ref{cor:ROL_AB_implies_possible_ROLs}.
 \begin{corollary}\label{cor:12ROL_AB_implies_possible_12ROLs}
     Let $A\in\bK^{m\times n}$ and $B\in\bK^{n\times k}$ be two matrices.
     Then
     \begin{enumerate}[i)]
         \item $\pinv{\brac*{B^*}}\pinv{\brac*{AB}}\in \minv{(AB)B^*}{1,2}$,
         \item $B\pinv{\brac*{AB}}\in \minv{(AB)\pinv{B}}{1,2}$,
         \item $\pinv{\brac*{AB}}\pinv{\brac*{A^*}}\in \minv{A^*(AB)}{1,2}$,
         \item $\pinv{\brac*{AB}}A\in \minv{\pinv{A}(AB)}{1,2}$,
     \end{enumerate}
     Moreover, if $\pinv{B}\pinv{A}\in \minv{(AB)}{1,2}$, then
     \begin{enumerate}[i)]
         \item $\pinv{B}\pinv{\brac*{A^*A}}\in \minv{\brac*{A^*A}B}{1,2}$,
         \item $\pinv{\brac*{BB^*}}\pinv{A}\in \minv{A\brac*{BB^*}}{1,2}$,
         \item $\pinv{B}\brac*{\pinv{A}A}\in \minv{\brac*{\pinv{A}A}B}{1,2}$,
         \item $\brac*{B\pinv{B}}\pinv{A}\in \minv{A\brac*{B\pinv{B}}}{1,2}$,
     \end{enumerate}
     cf.~Corollary~\ref{cor:ROL_AB_implies_possible_ROLs}.
 \end{corollary}

 \begin{proof}
     For the first part i)
     \[\rng{(AB)^*}=\rng{B^*A^*}\subset \rng{B^*},\]
     so that all principal angles between $\rng{(AB)^*}$ and $\rng{B^*}$ are zero, 
     etc.
     For the second part i), if $\pinv{B}\pinv{A}\in \minv{(AB)}{1,2}$  then
     \[\rng{\brac{A^*A}^*}=\rng{A^*A}=\rng{A^*},\]
     hence its principal angles with $\rng{B}$ are $0$ or $\frac{\pi}{2}$, etc.
 \end{proof}

Using similar techniques one can prove the following equivalent conditions, some of them known before, cf.~\cite[Theorem 95~a)]{Tian512}
\begin{theorem}\label{thm:123_124_equiv}
 Let $A\in\bK^{m\times n}$ and $B\in\bK^{n\times k}$ be two matrices. 
    Then Table~\ref{tab:similar_123_124_equivalent} gives equivalent conditions for $\pinv{B}\pinv{A}$ to be an $\{1,2,3\}$- or $\{1,2,4\}$-inverse of $AB$.
    \renewcommand{\arraystretch}{1.4}
\begin{table}[ht!]
    \centering
    \begin{tabular}{c|c}
       $\pinv{B}\pinv{A}$ is a $\{1,2,3\}$-inverse of $AB$  & $\pinv{B}\pinv{A}$ is a $\{1,2,4\}$-inverse of $AB$ \\ \hline
       $\rng{A^*AB}=\rng{A^*}\cap\rng{B}$  &  $\rng{BB^*A^*}=\rng{A^*}\cap\rng{B}$\\
       $AB\pinv{B}\pinv{A}$ is an orthogonal projection &  $\pinv{B}\pinv{A}AB$ is an orthogonal projection \\
       $AB\pinv{B}\pinv{A}=AB\pinv{\brac*{AB}}$  &  $\pinv{B}\pinv{A}AB=\pinv{\brac*{AB}}AB$ \\
      \makecell{there exist SVD decompositions of $A,B$  \\ and an orthogonal matrix $Q\in\bK^{r_B\times r_B}$ such that \\$V_A^*U_B=\left[\begin{array}{cc}
            0 & *\\
            Q & 0\\
            0 & *\\
        \end{array}\right]$ or $AB=0$}  & \makecell{there exist SVD decompositions of $A,B$  \\ and an orthogonal matrix $Q\in\bK^{r_A\times r_A}$ such that \\$V_A^*U_B=\left[\begin{array}{ccc}
            0 & Q & 0\\
            * & 0 & *\\
        \end{array}\right]$ or $AB=0$} \\
       $\{1,2,3\}$-ROL holds for ${A^*A},BB^*$ &  $\{1,2,4\}$-ROL holds for ${A^*A},BB^*$ \\
      \makecell{the principal angles between $\rng{A^*}$ and $\rng{B}$ \\ belong to the set $\bracc*{0,\frac{\pi}{2}}$ and  $\rng{A^*}\cap\rng{B}$ \\  is spanned by left singular vectors of $A$}  &  \makecell{the principal angles between $\rng{A^*}$ and $\rng{B}$ \\ belong to the set $\bracc*{0,\frac{\pi}{2}}$ and  $\rng{A^*}\cap\rng{B}$ \\  is spanned by right singular vectors of $B$}\\

    \end{tabular}
    \caption{Similarity between equivalent conditions for $\{1,2,3\}$ and $\{1,2,4\}$ ROLs}
    \label{tab:similar_123_124_equivalent}
\end{table}
\end{theorem}

\begin{proof}
Only for $\{1,2,3\}$-ROL. We omit most of the details, as they are similar to $\{1,2\}$ and $\{1,2,3,4\}$ cases.

If $\pinv{B}\pinv{A}\in\minv{(AB)}{1,2,3}$ then $AB\pinv{B}\pinv{A}$ is a Hermitian matrix and, by Theorem~\ref{thm:12_eqv} point $viii)$ $\pinv{A}A$ and $B\pinv{B}$ commute, therefore
\[\brac*{AB\pinv{B}\pinv{A}}\brac*{AB\pinv{B}\pinv{A}}=A\brac*{B\pinv{B}}\brac*{\pinv{A}A}\brac*{B\pinv{B}}\pinv{A}=
A\brac*{\pinv{A}A}\brac*{B\pinv{B}}\brac*{B\pinv{B}}\pinv{A}=
AB\pinv{B}\pinv{A}.\]
Therefore $AB\pinv{B}\pinv{A}$ is an orthogonal projection.

If $AB\pinv{B}\pinv{A}$ is an orthogonal projection (which implies that
$\pinv{B}\pinv{A}\in\minv{(AB)}{3}$) then
\[\brac*{AB\pinv{B}\pinv{A}}\brac*{AB\pinv{B}\pinv{A}}=AB\pinv{B}\pinv{A}.\]
Multiplying the above equation on the left by $\pinv{A}$ and by $A$ on the right we have
\[\brac*{\pinv{A}AB\pinv{B}\pinv{A}A}\brac*{\pinv{A}AB\pinv{B}\pinv{A}A}=\pinv{A}AB\pinv{B}\pinv{A}A,\]
therefore $\pinv{A}AB\pinv{B}\pinv{A}A$ is an orthogonal projection. By Lemma~\ref{lem:sing_values_1_of_proj} its singular values/eigenvalues are $0$ and $1$. By Theorem~\ref{thm:12_eqv} case $xii)$ $\pinv{B}\pinv{A}\in\minv{(AB)}{1,2}$ as well.

If $AB\pinv{B}\pinv{A}=AB\pinv{\brac*{AB}}$ then $AB\pinv{B}\pinv{A}$ is an orthogonal projection.
If $\pinv{B}\pinv{A}\in\minv{(AB)}{1,2,3}$ then $AB\pinv{B}\pinv{A}$ is an orthogonal projection
and 
\[\rng{AB\pinv{B}\pinv{A}}\subset \rng{AB}.\]
On the other hand, $AB\pinv{B}\pinv{A}AB=AB$ therefore $\rank AB\pinv{B}\pinv{A}\ge\rank AB$, so 
$\rng{AB\pinv{B}\pinv{A}}=\rng{AB}$. It follows that $AB\pinv{B}\pinv{A}=AB\pinv{\brac*{AB}}$ as it is an orthogonal projection onto $\rng{AB}$.

The principal angles between 
\[\rng{A^*}\cap\rng{B}=\rng{A^*A}\cap\rng{BB^*},\]
are the same, and $A$ with $A^*A$ share the same left singular vectors.
\end{proof}

%% file: matlab.tex
The following code in MATLAB (R2022a) completes a random $m$-by-$n$ matrix $A$ of given rank $r_A$ with a $n$-by-$k$ matrix $B$ of rank $r_B\le \min(n,k)$ such that $\pinv{\brac*{AB}}=\pinv{B}\pinv{A}$. The notation is consistent with the one used in Section~\ref{sec:sufficient}.

\begin{lstlisting}
% see for yourself 

m = 40; n = 21; k = 30;
rA = 6;  % rank of A, not greater than m and n
rB = 8;  % rank of B, not greater than n and k
N = 3; % rank AB = dim(C(A*)\cap C(B)),  not greater than rA and rB

% fix matrix A
SA = zeros(m,n);
SA(1:m+1:rA*m+1) = randi(rA,[1 rA]); % assign arbitrary singular values
[UA,~] = qr(rand(m)); [VA,~] = qr(rand(n)); % random orthogonal matirices
A = UA*SA*VA'; % matrix A of rank rA

% construct matrix B
[Q,~] = qr(rand(N));
UB = [VA(:,1:N)*Q VA(:,(rA+1):((rA+1)+rB-N)) VA(:,(N+1):rA) VA(:,((rA+1)+rB-N+1):end)];
[VB,~] = qr(rand(k));
SB = zeros(n,k);
SB(1:n+1:rB*n+1) = randi(rB,[1 rB]); % assign arbitrary singular values
B = UB*SB*VB'; 

norm(pinv(A*B) - pinv(B)*pinv(A))
rank(B) - rB
size(B) - [n k]
rank(A*B) - N

% tests

AA=A'*A; BB=B*B';
pAA = pinv(A)*A; ApA = A*pinv(A);
pBB = pinv(B)*B; BpB = B*pinv(B);  

% see Lemma §\ref{lem:ROL_AB_implies_ROL_AA_BB}§
norm(pinv(AA*BB)-pinv(BB)*pinv(AA))

% see Corollary §\ref{cor:ROL_AB_implies_possible_ROLs}§
norm(pinv(A'*A*B)-pinv(A*B)*pinv(A'))+...
norm(pinv(pinv(A)*A*B)-pinv(A*B)*A)+... % pinv(pinv(A)) = A
norm(pinv(A*B*B')-pinv(B')*pinv(A*B))+...
norm(pinv(A*B*pinv(B))-B*pinv(A*B))+... % pinv(pinv(B)) = B
norm(pinv(AA*B)-pinv(B)*pinv(AA))+... %!!! 
norm(pinv(A*BB)-pinv(BB)*pinv(A))+... 
norm(pinv(pAA*B)-pinv(B)*pAA)+... % pinv(pAA) = pAA 
norm(pinv(A*BpB)-BpB*pinv(A))+... % pinv(BpB) = BpB
norm(pinv(B'*pinv(A))-A*pinv(B'))+...
norm(pinv(pinv(B)*A')-pinv(A')*B)

% see Corollary §\ref{cor:image_of_AB_under_1234}§
P1 = A*B*pinv(A*B);
P2 = (A*A'*A*B)*pinv(A*A'*A*B);
P3 = (A*B*B'*A')*pinv(A*B*B'*A');
CBcapCAs = null([null(B')'; null(A)']);
P4 = CBcapCAs*pinv(CBcapCAs);
norm(P1-P2) + norm(P1-P3)

% see Theorem §\ref{thm:pAA_BB_and_AA_BpB_commute}§
norm(pAA*BB-BB*pAA) + norm(AA*BpB-BpB*AA)

% see Remark §\ref{rem:range_of_Grevilles_spaces}§
% by prinicipal angles
CBcapCAs = null([null(B')'; null(A)']);
QAAB = orth(A'*A*B);
QBBA = orth(B*B'*A');
Qcap = orth(CBcapCAs);
svd(QAAB'*QBBA)
svd(Qcap'*QBBA)
\end{lstlisting}

The following returns two matrices $A,B$ such that $\pinv{B}\pinv{A}\in \minv{(AB)}{1,2}$ using the property that is equivalent to the principal angles between $\rng{A^*}$ and $\rng{B}$ being equal to $0$ or $\frac{\pi}{2}$.

\begin{lstlisting}
% see for yourself 

m = 40; n = 21; k = 30;
rA = 6;  % rank of A, not greater than m and n
rB = 8;  % rank of B, not greater than n and k
N = 3; % rank AB = dim(C(A*) \cap C(B)),  not greater than rA and rB

CBcapCAs = orth(rand(n,N)); VA = zeros(n,n); UB = zeros(n,n);
VA(:,1:N) = CBcapCAs; UB(:,1:N) = CBcapCAs;
VA(:,(N+1):n) = orth(null(CBcapCAs'));
NAs = orth(null(VA(:,1:rA)'));
UB(:,(N+1):rB) = NAs(:,1:rB-N);
UB(:,(rB+1):n) = orth(null(UB(:,1:rB)'));
svd((VA(:,1:rA))'*UB(:,1:rB)) % check the angles
VA(:,1:rA) = VA(:,1:rA) * orth(rand(rA)); % change bases of C(A*) and C(B)
UB(:,1:rB) = UB(:,1:rB) * orth(rand(rB));
SA = zeros(m,n); SB = zeros(n,k);
SA(1:m+1:rA*m+1) = randi(rA,[1 rA]); % assign arbitrary singular values
SB(1:n+1:rB*n+1) = randi(rB,[1 rB]); % assign arbitrary singular values
[UA,~] = qr(rand(m)); [VB,~] = qr(rand(k)); % random orthogonal matirices
A = UA*SA*VA'; % matrix A of rank rA
B = UB*SB*VB'; % matrix B of rank rB

pBpA = pinv(B)*pinv(A);
AB = A*B;
% pinv(B)pinv(A) is 1,2-inverse of AB
norm(AB*pBpA*AB - AB) + norm(pBpA*AB*pBpA - pBpA)
% but it is not in general the 1,2,3,4-inverse of AB
norm(pinv(AB)-pBpA)

% auxiliary function

test = @(A,B) norm(pinv(B)*pinv(A)*A*B*pinv(B)*pinv(A)-pinv(B)*pinv(A))+...
norm(A*B*pinv(B)*pinv(A)*A*B-A*B)

% tests

% see Corollary §\ref{cor:12ROL_AB_implies_possible_12ROLs}§
test(A*B,B') + test(A*B,pinv(B)) + test(A',A*B) + test(pinv(A),A*B) +...
test(A'*A,B) + test(A,B*B') + test(pinv(A)*A,B) + test(A,B*pinv(B)) + ...
test(pinv(B),A') + test(B',pinv(A))

% see Corollary §\ref{cor:other_12ROLs}§
test(A'*A,B*B')
test(pinv(A)*A,B*pinv(B))

% see Theorem §\ref{thm:12_eqv}§
P = pinv(A)*A*B*pinv(B); Q = B*pinv(B)*pinv(A)*A;
norm(P^2-P) + norm(Q^2-Q) + norm(P'-P) + norm(Q'-Q)
eig(pinv(A)*A*B*pinv(B)*pinv(A)*A)
\end{lstlisting}

It is interesting to observe that the difference between $\pinv{(AB)}$ and $\pinv{B}\pinv{A}$ is usually not large.

The following code illustrates Theorem~\ref{thm:123_124_equiv} in the $\{1,2,3\}$ case.

\begin{lstlisting}
% see for yourself 

m = 40; n = 21; k = 30;
rA = 6;  % rank of A, not greater than m and n
rB = 8;  % rank of B, not greater than n and k
N = 3; % rank AB = dim(C(A*) \cap C(B)),  not greater than rA and rB

VA = orth(rand(n)); 
CB = VA(:,(rA-N+1):(rA-N+rB))*orth(rand(rB));
NBs = orth(null(CB'));
UB = [CB NBs];
UA = orth(rand(m)); VB = orth(rand(k));
SA = zeros(m,n); SB = zeros(n,k);
SA(1:m+1:rA*m+1) = randi(rA,[1 rA]); % assign arbitrary singular values
SB(1:n+1:rB*n+1) = randi(rB,[1 rB]); % assign arbitrary singular values
A = UA*SA*VA'; % matrix A of rank rA
B = UB*SB*VB'; % matrix B of rank rB

% auxiliary function

test = @(A,B) norm(pinv(B)*pinv(A)*A*B*pinv(B)*pinv(A)-pinv(B)*pinv(A))+...
norm(A*B*pinv(B)*pinv(A)*A*B-A*B) + ...
norm(A*B*pinv(B)*pinv(A) - (A*B*pinv(B)*pinv(A))')

test(A,B)
test(A'*A,B*B')

% 124 not satisfied
norm(pinv(B)*pinv(A)*A*B - (pinv(B)*pinv(A)*A*B)')

% other conditions
norm(A*B*pinv(B)*pinv(A)*A*B*pinv(B)*pinv(A)-A*B*pinv(B)*pinv(A))
norm(A*B*pinv(B)*pinv(A)-(A*B*pinv(B)*pinv(A))')

norm(A*B*pinv(B)*pinv(A)-A*B*pinv(A*B))

\end{lstlisting}